\documentclass[11pt]{amsart}
\usepackage{mathtools} \mathtoolsset{showonlyrefs,showmanualtags}
\usepackage[
    colorlinks=true,       
    linkcolor=magenta,          
    citecolor=magenta,        
    filecolor=magenta,      
    urlcolor=magenta,           
    pagebackref=true]{hyperref}

\usepackage{stmaryrd,mathrsfs}
\usepackage{xcolor}
\numberwithin{equation}{section}
\allowdisplaybreaks[4]
\usepackage{amsmath,amsthm,amscd,amssymb,amsfonts, amsbsy}
\usepackage{latexsym}
\usepackage{exscale}
\newtheorem{Theorem}[equation]{Theorem}
\newtheorem{nonum}{Theorem}

\newtheorem{Proposition}[equation]{Proposition}

\newtheorem{Lemma}[equation]{Lemma}

\newtheorem{Remark}[equation]{Remark}

\def\Xint#1{\mathchoice
{\XXint\displaystyle\textstyle{#1}}%
{\XXint\textstyle\scriptstyle{#1}}%
{\XXint\scriptstyle\scriptscriptstyle{#1}}%
{\XXint\scriptscriptstyle\scriptscriptstyle{#1}}%
\!\int}
\def\XXint#1#2#3{{\setbox0=\hbox{$#1{#2#3}{\int}$}
\vcenter{\hbox{$#2#3$}}\kern-.5\wd0}}

\def\dashint{\Xint-}

\def\bbR{\mathbb{R}}
\def\Rn{\mathbb{R}^n}

\def\d{\textup{d}}

\begin{document}

\title[Sparse domination theorem]{Sparse domination theorem for multilinear singular integral operators with $L^{r}$-H\"ormander condition}

\author{Kangwei Li}

\address{Department of Mathematics and Statistics, P.O.B. 68 (Gustaf
H\"all\-str\"omin katu 2b), FI-00014 University of Helsinki, Finland}
\curraddr{BCAM, Basque Center for Applied Mathematics, Mazarredo, 14. 48009 Bilbao Basque Country, Spain}
\email{kangwei.nku@gmail.com, kli@bcamath.org}

\thanks{The author was supported by the European Union through Tuomas Hyt\"onen's ERC
Starting Grant
``Analytic-probabilistic methods for borderline singular integrals''.
He was a member of the Finnish Centre of Excellence in Analysis and
Dynamics Research. He is also supported by the Basque Government through the BERC 2014-2017 program and
by Spanish Ministry of Economy and Competitiveness MINECO: BCAM Severo Ochoa excellence accreditation SEV-2013-0323.
}

\date{\today}


\begin{abstract}
In this note, we show that if $T$ is a multilinear singular integral operator associated with a kernel satisfies the so-called multilinear $L^{r}$-H\"ormander condition, then $T$ can be dominated by  multilinear sparse operators.
\end{abstract}

\maketitle

\section{Introduction and main results}\label{Sect:1}
This note is devoted to obtain a sparse domination formula for a class of multilinear singular integral operators. 
 Dominating Calder\'on-Zygmund operators by sparse operators starts from Lerner \cite{Ler-imrn}, in which paper he obtained the following result,
 \begin{equation}
\|Tf(x)\|_X\le C_T \sup_{\mathcal D, \mathcal S}\Big\| \sum_{Q\in \mathcal S} \Big( \dashint_Q |f|\Big) \chi_Q(x)\Big\|_X,
\end{equation}where the supremum is taken over all the sparse families $\mathcal S\subset \mathcal D$ (see below for the definition) and all the dyadic grids $\mathcal D$ and $X$ is an arbitrary Banach function space. Then the $A_2$ theorem (due to Hyt\"onen \cite{Hyt}) follows as an easy consequence.

Later, this result was refined by a pointwise control independently and simultaneously by Conde-Alonso and Rey \cite{CR}, and by Lerner and Nazarov \cite{LN}. All the results mentioned above require the kernel satisfying the $\log$-Dini condition. Finally, Lacey \cite{Lac15} relaxed the $\log$-Dini condition to Dini condition and then Hyt\"onen, Roncal and Tapiola \cite{HRT}
refined the proof by tracking the precise dependence on the constants. Very recently, Lerner \cite{Ler2016} also provided a new proof for this result, which also works for some more general operators.
To be precise, Lerner showed the following result:
 \begin{nonum}\cite[Theorem 4.2]{Ler2016}\label{thm:lerner}
Given a sublinear operator $T$. Assume that $T$ is of weak type $(q, q)$ and the corresponding grand maximal truncated operator $\mathcal M_T$ is of weak type $(r,r)$, where $1\le q\le r<\infty$. Then, for every compactly supported $f\in L^r(\bbR^n)$, there exists a sparse family $\mathcal S$ such that for a.e. $x\in \bbR^n$,
\[
|T(f)(x)|\le c_{n,q,r}(\|T\|_{L^q\rightarrow L^{q,\infty}}+ \|\mathcal M_T\|_{L^r\rightarrow L^{r,\infty}})\sum_{Q\in \mathcal S}\Big( \dashint_Q |f|^r\Big)^{\frac 1r}\chi_Q(x).
\]
\end{nonum}
Let us recall the notations in the above result.
We say $\mathcal S$ is a sparse family if for all cubes $Q\in \mathcal S$, there exists $E_Q\subset Q$ which are pairwise disjoint and $|E_Q|\ge \gamma |Q|$, where $0<\gamma<1$. For a given operator $T$, the so-called ``grand maximal truncated'' operator $\mathcal M_T$ is defined by
\[
\mathcal M_T f(x)=\sup_{Q\ni x}\mathop{\text{ess\,sup}}_{\xi \in Q}|T(f\chi_{\bbR^n\setminus 3Q})(\xi)|.
\]

It is shown in \cite{Ler2016} that Calder\'on-Zygmund operators with Dini continuous kernel satisfy the assumption in Theorem~\ref{thm:lerner} with $q=r=1$. In this short note, we \nobreak{shall} give a multilinear analogue of Theorem \ref{thm:lerner}. Then as an application, we give a sparse domination formula for multilinear singular integral operators whose kernel $K(x,y_1,\cdots,y_m)$ satisfies the so-called $m$-linear $L^{r}$-H\"ormander condition
\begin{align}\label{eq:hormander}
K_r:=\sup_Q \sup_{x,z\in \frac 12 Q}\sum_{k=1}^{\infty}&|2^k Q|^{\frac m r}\Big(\int_{(2^{k}Q )^m\setminus (2^{k-1} Q)^m}\\
&|K(x,y_1,\cdots,y_m)-K(z,y_1,\cdots,y_m)|^{r'}\d {\vec y}\Big)^{\frac 1{r'}}<\infty,\nonumber
\end{align}
where $Q^m=\underbrace{Q\times \cdots \times Q}_{m\,\text{times}}$ and $1\le r<\infty$. When $r=1$, the above formula is understood as 
\begin{align*}
K_1:=\sup_Q &\sup_{x,z\in \frac 12 Q}\sum_{k=1}^{\infty}|2^k Q|^{  m  }\\
&\mathop{\text{ess\,sup}}_{\vec y\in (2^{k}Q )^m\setminus (2^{k-1} Q)^m} |K(x,y_1,\cdots,y_m)-K(z,y_1,\cdots,y_m)|<\infty.
\end{align*}
In \cite{BFP}, Bernicot, Frey and Petermichl showed that  a large class of singular non-integral operators can be dominated by sparse operators (actually in the bilinear form sense). Even in the linear case, the $L^r$-H\"ormander condition is beyond the ``off-diagonal estimate'' assumption in \cite{BFP}. We also remark that our assumption is weaker than the assumption  (H2)  used in \cite{BCDH} (see Proposition \ref{prop:h2}). It is also easy to see that our assumption is weaker than the Dini condition used in \cite{DHL} (see Proposition \ref{prop:dini}).

Now to state our main result, we need  a multilinear analogue of grand maximal truncated operator. Given an operator $T$, define
\[
\mathcal M_T (f_1,\cdots, f_m)(x)=\sup_{Q\ni x}\mathop{\text{ess\,sup}}_{\xi \in Q}|T(f_1,\cdots, f_m)(\xi)-T(f_1\chi_{3Q},\cdots, f_m\chi_{3Q})(\xi) |,
\]note that we don't require $T$ to be multi-sublinear.
Given a cube $Q_0$, for $x\in Q_0$, we also define a local version of $\mathcal M_T$ by
\begin{align*}
\mathcal M_{T,Q_0} (f_1,\cdots, f_m)(x)=\sup_{Q\ni x, Q\subset Q_0}\mathop{\text{ess\,sup}}_{\xi \in Q}|T&(f_1\chi_{3Q_0},\cdots, f_m\chi_{3Q_0})(\xi)\\
&-T(f_1\chi_{3Q},\cdots, f_m\chi_{3Q})(\xi) |.
\end{align*}
Our first result reads as follows
\begin{Theorem}\label{thm:m1}
Assume that $T$ is bounded from $L^{q}\times \cdots \times L^q$ to $L^{q/m,\infty}$ and $\mathcal M_T$ is bounded from $L^r\times \cdots \times L^r$ to $L^{r/m, \infty}$, where $1\le q\le r<\infty$. Then, for compactly supported functions $f_i\in L^r(\bbR^n)$, $i=1,\cdots m$, there exists a sparse family $\mathcal S$ such that for a.e. $x\in \bbR^n$,
\begin{align*}
|T(f_1,\cdots, f_m)(x)|\le c_{n,q,r}&(\|T\|_{L^q\times\cdots\times L^q\rightarrow L^{q/m,\infty}}+ \|\mathcal M_T\|_{L^r\times \cdots\times L^r\rightarrow L^{r/m,\infty}})\\
&\times\sum_{Q\in \mathcal S}\prod_{i=1}^m\Big( \dashint_Q |f_i|^r\Big)^{\frac 1r}\chi_Q(x).
\end{align*}
\end{Theorem}
As a consequence, we show the following
\begin{Theorem}\label{thm:m2}
Let $T$ be a multilinear singular integral operator which is bounded from  $L^{r}\times \cdots \times L^r$ to $L^{r/m,\infty}$ and its kernel satisfies the $m$-linear $L^{r}$-H\"ormander condition.
Then, for compactly supported functions $f_i\in L^r(\bbR^n)$, $i=1,\cdots m$, there exists a sparse family $\mathcal S$ such that for a.e. $x\in \bbR^n$,
\begin{align*}
|T(f_1,\cdots, f_m)(x)|\le c_{n,r}&(\|T\|_{L^r\times\cdots\times L^r\rightarrow L^{r/m,\infty}}+ K_{r})\\
&\times\sum_{Q\in \mathcal S}\prod_{i=1}^m\Big( \dashint_Q |f_i|^r\Big)^{\frac 1r}\chi_Q(x).
\end{align*}
\end{Theorem}
We would like to remark that recently a different approach to sparse domination of singular integrals not relying on weak endpoint bounds for grand maximal functions has been developed by many authors, see \cite{CACDPO, CDPO, KL, LacS}.

In the next section, we shall give a proof for Theorems \ref{thm:m1} and \ref{thm:m2}. And in Section \ref{sect:3}, we will give some remarks about the $m$-linear $L^r$-H\"ormander condition.

\section{Proof of Theorems \ref{thm:m1} and \ref{thm:m2}}
The proof of Theorem \ref{thm:m1} will follow the same idea used in \cite{Ler2016}, but with proper changes to make it suit for the multilinear case. For simplicity, we only prove both results in the case of $m=2$, and the general case can be proved similarly. First, we prove the following lemma.
\begin{Lemma}\label{lm:grand}
Suppose $T$ is  bounded from $L^q\times L^q\rightarrow L^{q/2,\infty}$, $1\le q<\infty$.
Then for a.e. $x\in Q_0$,
\[
|T(f_1\chi_{3Q_0}, f_2\chi_{3Q_0})(x)|\le c_n\|T\|_{L^q\times L^q\rightarrow L^{\frac q2,\infty}}|f_1(x)f_2(x)|+ \mathcal M_{T, Q_0}(f_1, f_2)(x).
\]
\end{Lemma}
\begin{proof}
Suppose that $x\in \text{int} Q_{0}$ and let $x$ be
a point of approximate continuity of $T(f_1\chi_{3Q_0}, f_2\chi_{3Q_0})$ (see e.g. \cite[p.46]{EG}). Then for every $\varepsilon >0$, the sets
\[
E_s(x):=\{y\in B(x,s): |T(f_1\chi_{3Q_0}, f_2\chi_{3Q_0})(y)-T(f_1\chi_{3Q_0}, f_2\chi_{3Q_0})(x)|<\varepsilon\}
\]
satisfy $\lim_{s\rightarrow 0}\frac{|E_s(x)|}{|B(x,s)|}=1$. Denote by $Q(x, s)$ the smallest cube centered at $x$ and containing $B(x,s)$. Let $s > 0$ be so small that $Q(x, s) \subset Q_0$. Then for a.e. $y\in E_s(x)$,
\begin{align*}
|T(f_1\chi_{3Q_0}, f_2\chi_{3Q_0})(x)|&< |T(f_1\chi_{3Q_0}, f_2\chi_{3Q_0})(y)|+\varepsilon\\
& \le  |T(f_1\chi_{3Q(x,s)}, f_2\chi_{3Q(x,s)})(y)|+\mathcal M_{T, Q_0}(f_1, f_2)(x)+\varepsilon.
\end{align*}
It follows that
\begin{align*}
&|T(f_1\chi_{3Q_0}, f_2\chi_{3Q_0})(x)|\\
&\le\mathop{\textup{ess\,inf}}_{y\in E_s(x)} |T(f_1\chi_{3Q(x,s)}, f_2\chi_{3Q(x,s)})(y)|+\mathcal M_{T, Q_0}(f_1, f_2)(x)+\varepsilon \\
&\le  |E_s(x)|^{-\frac 2q}\|T(f_1\chi_{3Q(x,s)}, f_2\chi_{3Q(x,s)})\|_{L^{q/2,\infty}}
 + \mathcal M_{T, Q_0}(f_1, f_2)(x)+\varepsilon\\
&\le \|T\|_{L^q\times L^q\rightarrow L^{\frac q2,\infty}}\frac 1{|E_s(x)|^{2/q}}\prod_{i=1}^2\Big(\int_{3Q(x,s)} |f_i(y)|^q\d y\Big)^{\frac 1q}
 + \mathcal M_{T, Q_0}(f_1, f_2)(x)+\varepsilon.
\end{align*}
Assuming additionally that $x$ is a Lebesgue point of $|f_1|^q$ and $|f_2|^q$ and letting
subsequently $s\rightarrow 0$ and $\varepsilon \rightarrow 0$ will conclude the proof.
\end{proof}
Now we are ready to prove Theorem \ref{thm:m1}.
\begin{proof}[Proof of Theorem \ref{thm:m1}]
 Fix a cube $Q_0\subset \Rn$. We shall prove the following recursive inequality,
\begin{equation}\label{recursive}
|T(f_1\chi_{3Q_0}, f_2\chi_{3Q_0})(x)|\chi_{Q_0}\le c_{n,T} \langle |f_1|^r\rangle_{3Q_0}^{\frac 1r} \langle |f_2|^r\rangle_{3Q_0}^{\frac 1r}+ \sum_j |T(f_1\chi_{3P_j}, f_2\chi_{3P_j})(x)|\chi_{P_j},
\end{equation}
where $P_j$ are disjoint dyadic subcubes of $Q_0$, i.e. $P_j\in \mathcal D(Q_0)$ and moreover, $\sum_j |P_j|\le \frac 12 |Q_0|$.
Observe that for arbitrary pairwise disjoint cubes $P_j\in \mathcal D(Q_0)$, we have
\begin{align*}
|T(f_1\chi_{3Q_0}, &f_2\chi_{3Q_0})(x)|\chi_{Q_0}\\&=|T(f_1\chi_{3Q_0}, f_2\chi_{3Q_0})(x)|\chi_{Q_0\setminus \cup_j P_j}+\sum_j|T(f_1\chi_{3Q_0}, f_2\chi_{3Q_0})(x)|\chi_{P_j}\\
&\le |T(f_1\chi_{3Q_0}, f_2\chi_{3Q_0})(x)|\chi_{Q_0\setminus \cup_j P_j}+\sum_j|T(f_1\chi_{3P_j}, f_2\chi_{3P_j})(x)|\chi_{P_j}\\
&+ \sum_j|T(f_1\chi_{3Q_0}, f_2\chi_{3Q_0})(x)-T(f_1\chi_{3P_j}, f_2\chi_{3P_j})(x)|\chi_{P_j}.
\end{align*}
Hence, in order to prove the recursive claim, it suffices to show that
one can select pairwise disjoint cubes $P_j\in \mathcal D(Q_0)$ with $\sum_j |P_j|\le \frac 12 |Q_0|$ and such that for a.e. $x\in Q_0$,
\begin{align*}
&\sum_j|T(f_1\chi_{3Q_0}, f_2\chi_{3Q_0})(x)-T(f_1\chi_{3P_j}, f_2\chi_{3P_j})(x)|\chi_{P_j}\\
&\quad+ |T(f_1\chi_{3Q_0}, f_2\chi_{3Q_0})(x)|\chi_{Q_0\setminus \cup_j P_j}\le c_{n,T} \langle |f_1|^r\rangle_{3Q_0}^{\frac 1r}\langle |f_2|^r\rangle_{3Q_0}^{\frac 1r}.
\end{align*}
By our assumption, $\mathcal M_T$ is bounded from $L^r\times L^r$ to $L^{r/2,\infty}$.
Therefore, there is some sufficient large $c_n$ such that the set
\begin{align*}
E:= &\{x\in Q_0: |f_1(x)f_2(x)|>c_n \langle |f_1|^r\rangle_{3Q_0}^{\frac 1r}\langle |f_2|^r\rangle_{3Q_0}^{\frac 1r}\}\\
&\cup \{x\in Q_0: \mathcal M_{T, Q_0}(f_1, f_2)(x)> c_n \|\mathcal M_T\|_{L^r\times L^r\rightarrow L^{r/2,\infty}}\langle |f_1|^r\rangle_{3Q_0}^{\frac 1r}\langle |f_2|^r\rangle_{3Q_0}^{\frac 1r}\}
\end{align*}
will satisfy $|E|\le \frac 1{2^{n+2}}|Q_0|$. The Calder\'on-Zygmund decomposition applied to the function $\chi_E$
on $Q_0$ at height $\lambda=\frac 1{2^{n+1}}$ produces pairwise disjoint cubes $P_j\in \mathcal D(Q_0)$
such that
\[
\frac 1{2^{n+1}}|P_j|\le |P_j\cap E|< \frac 12 |P_j|
\]
and $|E\setminus \cup_j P_j|=0$. It follows that $\sum_j |P_j|\le \frac 12 |Q_0|$ and $P_j\cap E^c\neq \emptyset$. Therefore,
\begin{align*}
\mathop{\textup{ess\,sup}}_{\xi \in P_j }&|T(f_1\chi_{3Q_0}, f_2\chi_{3Q_0})(\xi)-T(f_1\chi_{3P_j}, f_2\chi_{3P_j})(\xi)|\\
&\le c_n  \|\mathcal M_T\|_{L^r\times L^r\rightarrow L^{r/2,\infty}}\langle |f_1|^r\rangle_{3Q_0}^{\frac 1r}\langle |f_2|^r\rangle_{3Q_0}^{\frac 1r}.
\end{align*}
On the other hand, by Lemma \ref{lm:grand}, for a.e. $x\in Q_0\setminus \cup_j P_j$, we have
\begin{align*}
|T(f_1\chi_{3Q_0}, f_2\chi_{3Q_0})(x)|\le c_n (\|T\|_{L^q\times L^q \rightarrow L^{q/2,\infty}}&+ \|\mathcal M_T \|_{L^r\times L^r \rightarrow L^{r/2,\infty}})\\
&\times\langle |f_1|^r\rangle_{3Q_0}^{\frac 1r}\langle |f_2|^r\rangle_{3Q_0}^{\frac 1r}.
\end{align*}
Therefore, combining the estimates we arrive at \eqref{recursive} with
\[
c_{n,T}\eqsim \|T\|_{L^q\times L^q \rightarrow L^{q/2,\infty}}+  \|\mathcal M_T \|_{L^r\times L^r \rightarrow L^{r/2,\infty}}.
\]

Now with \eqref{recursive}, the rest of the argument is the same as that in \cite{Ler2016} and we complete the proof.
\end{proof}

Next we turn to prove Theorem \ref{thm:m2}
\begin{proof}[Proof of Theorem \ref{thm:m2}]
It suffices to prove that $\mathcal M_T$ is bounded from $L^r\times L^r$ to $L^{r/2,\infty}$. Indeed, let $x,x',\xi\in Q\subset \frac 12 \cdot 3Q$. We have
\begin{align}\label{eq:split}
&|T(f_1, f_2)(\xi)- T(f_1\chi_{3Q}, f_2\chi_{3Q})(\xi)| \\
&= \Big|\iint_{(\bbR^n)^2\setminus (3Q)^2} K(\xi, y_1, y_2)f_1(y_1) f_2(y_2)\d y_1 \d y_2 \Big| \nonumber\\
&\le \Big|\iint_{(\bbR^n)^2\setminus (3Q)^2}( K(\xi, y_1, y_2)-K(x',y_1, y_2) )f_1(y_1) f_2(y_2)\d y_1 \d y_2 \Big| \nonumber\\
&\quad + |T(f_1, f_2)(x')|+ |T(f_1\chi_{3Q}, f_2\chi_{3Q})(x')|. \nonumber
\end{align}
By the bilinear $L^r$-H\"ormander condition, we have
\begin{align*}
&\Big|\iint_{(\bbR^n)^2\setminus (3Q)^2}( K(\xi, y_1, y_2)-K(x',y_1, y_2) )f_1(y_1) f_2(y_2)\d y_1 \d y_2 \Big|\\
&\le \sum_{k=1}^\infty \iint_{(2^k3Q)^2\setminus (2^{k-1}3Q)^2} | K(\xi, y_1, y_2)-K(x',y_1, y_2) |\cdot|f_1(y_1)|\cdot| f_2(y_2)|\d y_1 \d y_2\\
&\le \sum_{k=1}^\infty \Big(\iint_{(2^k3Q)^2\setminus (2^{k-1}3Q)^2} | K(\xi, y_1, y_2)-K(x',y_1, y_2)|^{r'}\d y_1 \d y_2 \Big)^{\frac 1{r'}}\\
&\qquad \times \Big( \iint_{(2^k3Q)^2} |f_1(y_1)f_2(y_2)|^r \d y_1 \d y_2 \Big)^{\frac 1r}\\
&= \sum_{k=1}^\infty  |2^k3Q|^{\frac 2 r} \Big(\iint_{(2^k3Q)^2\setminus (2^{k-1}3Q)^2} | K(\xi, y_1, y_2)-K(x',y_1, y_2)|^{r'}\d y_1 \d y_2 \Big)^{\frac 1{r'}}\\
&\qquad \times \Big( \frac 1{|2^k3Q|^2}\iint_{(2^k3Q)^2} |f_1(y_1)f_2(y_2)|^r \d y_1 \d y_2 \Big)^{\frac 1r}\\
&\le K_r (\mathcal M(|f_1|^r, |f_2|^r)(x))^{\frac 1r}.
\end{align*}
Then by taking $L^{r/4}$ average over $x'\in Q$ on both side of  \eqref{eq:split} we obtain
\begin{align*}
&|T(f_1, f_2)(\xi)- T(f_1\chi_{3Q}, f_2\chi_{3Q})(\xi)| \\
&\le K_r (\mathcal M(|f_1|^r, |f_2|^r)(x))^{\frac 1r}+\Big( \frac 1{|Q|}\int_Q |T(f_1, f_2)(x')|^{\frac r4} \d x'\Big)^{\frac 4r}\\
&\quad + \Big( \frac 1{|Q|}\int_Q |T(f_1\chi_{3Q}, f_2\chi_{3Q})(x')|^{\frac r4} \d x'\Big)^{\frac 4r}\\
&\le K_r (\mathcal M(|f_1|^r, |f_2|^r)(x))^{\frac 1r}+ M_{r/4}(T(f_1, f_2))(x)\\
&\quad+ c_r\|T(f_1\chi_{3Q}, f_2\chi_{3Q})\|_{L^{\frac r 2,\infty}(Q, \frac{\d x'}{|Q|})}\\
&\le  K_r (\mathcal M(|f_1|^r, |f_2|^r)(x))^{\frac 1r}+ M_{r/4}(T(f_1, f_2))(x)\\
&\quad+ c_r\|T\|_{L^r\times L^r \rightarrow L^{r/2,\infty}} \Big( \frac 1{|Q|}\int_{3Q}|f_1|^r\Big)^{\frac 1r}\cdot \Big( \frac 1{|Q|}\int_{3Q}|f_2|^r\Big)^{\frac 1r}\\
&\le (K_r+c_{n,r}\|T\|_{L^r\times L^r \rightarrow L^{r/2,\infty}}) (\mathcal M(|f_1|^r, |f_2|^r)(x))^{\frac 1r}+ M_{r/4}(T(f_1, f_2))(x),
\end{align*}
where the bilinear maximal function $\mathcal M$ is defined as
\[
\mathcal M(f,g)(x)=\sup_{Q\ni x} \Big(\frac 1{|Q|}\int_Q |f_1| \Big)\Big(\frac 1{|Q|}\int_Q |f_2|\Big). 
\]
So we conclude that
\begin{align*}
\mathcal M_T(f_1, f_2)(x)\le (K_r&+c_{n,r}\|T\|_{L^r\times L^r \rightarrow L^{r/2,\infty}}) (\mathcal M(|f_1|^r, |f_2|^r)(x))^{\frac 1r}\\
&+ M_{r/4}(T(f_1, f_2))(x).
\end{align*}
It is obvious that $(\mathcal M(|f_1|^r, |f_2|^r)(x))^{\frac 1r}$ is bounded from $L^r\times L^r$ to $L^{r/2,\infty}$. On the other hand, since it is well-known that $M_{r/4}$ is bounded from $L^{r/2,\infty}$ to $L^{r/2,\infty}$, by the assumption that $T$ is bounded from $L^r\times L^r$ to $L^{r/2,\infty}$, we obtain
\[
\|M_{r/4}(T(f_1, f_2))\|_{L^{r/2,\infty}}\le c_{n,r} \|T\|_{L^r \times L^r \rightarrow L^{r/2,\infty}} \|f_1\|_{L^r}\|f_2\|_{L^r}.
\]

Therefore, $\mathcal M_T$ is bounded from $L^r\times L^r \rightarrow L^{r/2,\infty}$ and the desired result follows from Theorem \ref{thm:m1}.
\end{proof}

\section{Some remarks}\label{sect:3}
In this section, we give some remarks about the $L^r$-H\"ormander condition and some applications of our main result. It is well known that the H\"ormander condition
\[
\sup_{x,z\in \bbR^n}\int_{|y-x|> 2|x-z|}|K(x,y)-K(z,y)|\d y<\infty
\]
is not sufficient for
\begin{equation}\label{eq:CF}
\int_{\bbR^n} T(f)^p w(x) \d x\le C \int_{\bbR^n} M_r(f)^p w(x) \d x, \quad w\in A_\infty
\end{equation}
for any $r\ge 1$ and $0<p<\infty$ (see \cite[Theorem 3.1]{MPT}). So we cannot expect the sparse domination theorem for singular integral operators with this kernel. This is because once we have 
\[
|T(f)(x)|\le C_{T,r}\sum_{Q\in \mathcal S} \Big( \dashint_Q |f|^r\Big)^{\frac 1r}\chi_Q(x)
\] for some $r\ge 1$, then \eqref{eq:CF} holds for $p=1$ (see \cite[Lemma 4.1]{HP}) and therefore for all $0<p<\infty$ (see \cite[Theorem 1.1]{CMP}). Then it is reasonable to consider somewhat stronger condition such as our $L^r$-H\"ormander condition. For more background, see \cite{KW,MPT,W}.

Now we shall briefly show that our conditions are weaker than Dini condition, which is used in \cite{DHL} by Dami\'an, Hormozi and the author. Recall that the Dini condition is defined by
\begin{align*}\label{eq:smoothness}
    |K(x&+h,y_1,\cdots, y_m)-K(x,y_1,\cdots,y_m)| +   |K(x,y_1+h,\cdots,y_m)-K(x,y_1,\cdots,y_m)| \\   &+\cdots+|K(x,y_1,\cdots,y_m+h)-K(x,y_1,\cdots,y_m)| \\
    &\leq \frac{1}{(\sum_{i=1}^m|x-y_i|)^{2n}} \omega\left( \frac{|h|}{\sum_{i=1}^m|x-y_i|}\right),
\end{align*}
whenever $|h|\leq \frac 12\max\{|x-y_i|: i=1,\cdots,m\}$, where $\omega$ is increasing, $\omega(0)=0$ and $\|\omega\|_{\text{Dini}}=\int_0^1 w(t) \d t/ t<\infty$.
\begin{Proposition}\label{prop:dini}
$m$-linear Dini condition implies $m$-linear $L^r$-H\"ormander condition.
\end{Proposition}
\begin{proof}
Again, we just prove the case $m=2$. It is obvious that we just require regularity in the $x$ variable. Fix $x,z\in \frac 12 Q$. Since $|x-z|<\frac 12 \sqrt n \ell(Q)$, for $k>\log_2(1+ 4\sqrt n) $ and $(y_1, y_2)\in (2^k Q)^2\setminus (2^{k-1}Q)^2$, we have
$|x-z|\le \frac 12 \max\{|x-y_1|, |x-y_2|\}$. Therefore,
\begin{align*}
&\sum_{k>\log_2(1+ 4\sqrt n) }|2^k Q|^{\frac 2 r} \Big(\int_{(2^k Q)^2 \setminus (2^{k-1}Q)^2} |K(x,y_1,y_2)-K(z,y_1,y_2)|^{r'} \d \vec y\Big)^{\frac 1{r'}}\\
&\le \sum_{k>\log_2(1+ 4\sqrt n) }|2^k Q|^{\frac 2 r}\cdot w(\frac{2\sqrt n}{2^k-1})\cdot \frac 1{(2^{k-2}-\frac 14)^{2n}\ell(Q)^{2n}}\cdot |2^kQ|^{\frac 2{r'}}\\
&\lesssim_n \sum_{k>\log_2(1+ 4\sqrt n) } w(\frac{4\sqrt n}{2^k})\lesssim_n \|\omega\|_{\text{Dini}}.
\end{align*}
It remains to consider those $1\le k\le \log_2(1+ 4\sqrt n) $. We should be careful because we don't assume any size condition. In this case, since
\[
\max\{|y-y_1|, |y-y_2|\}\ge \frac 14 \ell(Q),\quad \forall \,y\in \frac 12 Q\,\,\mbox{and}\,(y_1, y_2)\in (2^k Q)^2\setminus (2^{k-1}Q)^2,
\]
we select $4\lceil \sqrt n \,\rceil$ points $x_1,\cdots x_{4\lceil \sqrt n \,\rceil}$ in the segment between $x$ and $z$ such that
\[
|x-x_1|, |x_i-x_{i+1}|, |x_{4\lceil \sqrt n \,\rceil}-z|\le \frac 1 8\ell(Q),\quad i=1,\cdots, 4\lceil \sqrt n \,\rceil-1.
\]
For convenience, denote $x_0=x$ and $x_{4\lceil \sqrt n \,\rceil+1}=z$. Then we have
\begin{align*}
|K(x,y_1,y_2)-K(z,y_1,y_2)|&\le \sum_{i=0}^{4\lceil \sqrt n \,\rceil}|K(x_i,y_1,y_2)-K(x_{i+1},y_1,y_2)|\\
&\le c_n \omega(\frac 12)\ell(Q)^{-2n}.
\end{align*}
Consequently,
\begin{align*}
&\sum_{1\le k\le\log_2(1+ 4\sqrt n) }|2^k Q|^{\frac 2 r} \Big(\int_{(2^k Q)^2 \setminus (2^{k-1}Q)^2} |K(x,y_1,y_2)-K(z,y_1,y_2)|^{r'} \d \vec y\Big)^{\frac 1{r'}}\\
&\lesssim_n \omega(\frac 12)\lesssim \|\omega\|_{\text{Dini}}.
\end{align*}
This completes the proof.
\end{proof}

Next we will show that $L^r$-H\"ormander condition is also weaker than the regularity assumption used in \cite{BCDH} (which was originally introduced in \cite{BD}). Recall that the regularity assumption in \cite{BCDH} reads as follows:

(\textbf{H2}): There exists $\delta>n/{r}$ so that
\begin{align*}
&\Big(  \int_{S_{j_m}(Q)}\cdots \int_{S_{j_1}(Q)}  |K(x, y_1,\cdots, y_m)-K(z,y_1,\cdots, y_m)|^{r'}  \d \vec{y} \Big)^{\frac 1{r'}}\\
&\qquad \le C\frac{|x-z|^{m(\delta-n/r)}}{|Q|^{m\delta/n}}2^{-m \delta j_0}
\end{align*}
for all cubes $Q$, all $x,z\in \frac 12 Q$ and $(j_1,\cdots, j_m)\neq (0,\cdots, 0)$, where $j_0=\max \{j_i\}_{1\le i\le m}$ and $S_j(Q)=2^j Q\setminus 2^{j-1}Q$ if $j\ge 1$, otherwise, $S_j(Q)=Q$.

\begin{Proposition}\label{prop:h2}
Assumption (H2) implies $m$-linear $L^r$-H\"ormander condition.
\end{Proposition}
\begin{proof}
Again, we just prove the bilinear case.
Observe that
\begin{align*}
(2^jQ)^2\setminus (2^{j-1}Q)^2&=(2^jQ\setminus 2^{j-1}Q)^2 \,\cup \,( 2^jQ\times (2^jQ\setminus 2^{j-1}Q )) \\
&\quad\cup\,( (2^jQ\setminus 2^{j-1}Q )\times 2^jQ)\\
&= (S_j(Q))^2\,\cup\, (\cup_{l\le j}S_l(Q)\times S_j(Q))\,\cup\, (\cup_{l\le j}S_j(Q)\times S_l(Q)).
\end{align*}
By triangle inequality, we have
\begin{align*}
&\sum_{j=1 }^\infty|2^j Q|^{\frac 2 r} \Big(\int_{(2^j Q)^2 \setminus (2^{j-1}Q)^2} |K(x,y_1,y_2)-K(z,y_1,y_2)|^{r'} \d \vec y\Big)^{\frac 1{r'}}\\
&\le \sum_{j=1 }^\infty|2^j Q|^{\frac 2 r} \Big(\int_{ (S_j(Q))^2} |K(x,y_1,y_2)-K(z,y_1,y_2)|^{r'} \d \vec y\Big)^{\frac 1{r'}}\\
& + \sum_{j=1 }^\infty|2^j Q|^{\frac 2 r}\sum_{l=0}^j \Big(\int_{  S_j(Q)\times S_l(Q)} |K(x,y_1,y_2)-K(z,y_1,y_2)|^{r'} \d \vec y\Big)^{\frac 1{r'}}\\
&  + \sum_{j=1 }^\infty|2^j Q|^{\frac 2 r}\sum_{l=0}^j \Big(\int_{  S_l(Q)\times S_j(Q)} |K(x,y_1,y_2)-K(z,y_1,y_2)|^{r'} \d \vec y\Big)^{\frac 1{r'}}\\
&\lesssim \sum_{j=1}^\infty |2^j Q|^{\frac 2 r} \frac 1{|Q|^{2/r}}2^{-2\delta j}(1+j)\\
&<\infty,
\end{align*}
the last inequality holds due to $\delta >n/r$. This completes the proof.
\end{proof}
\begin{Remark}
For $r>1$,  let $T$ be a linear Fourier multiplier with H\"ormander condition with parameter $n/2<s<n$ (see \eqref{eq:mh} in below for the definition). It is shown in   \cite[Theorem 3]{KW} that $T$ is not bounded on $L^p(w)$ for some $w\in A_p$ when $p<n/s$ or $p>(\frac ns)'$. This means  (H2) (which is a consequence of the assumption described as above, see \cite{BCDH}) and therefore the $L^r$-H\"ormander condition are not sufficient for the Dini condition. 

For $r=1$, recall that we only need the regularity on the  $x$-variable, in this sense, $L^1$-H\"ormander condition is strictly weaker than the Dini condition. However, the full regularity in the Dini condition is to ensure the weak endpoint boundedness. In fact, we can show that there is only a tiny difference between $L^1$-H\"ormander condition and Dini-condition in the $x$-variable. To see this, define 
\[
\omega^{x,z}(t):= \sup_{\vec y : \frac t2\le\frac{|x-z|}{\sum_{i=1}^2 |x-y_i| }\le t}|K(x,y_1,y_2)-K(z,y_1,y_2)| (\sum_{i=1}^2|x-y_i|)^{2n}.
\]
Then 
\begin{align*}
&\sup_{x,z}\sum_{k=1}^\infty \omega^{x,z}(2^{-k})\\
&\lesssim \sup_Q\sup_{x,z\in \frac 12 Q}\sum_{k=1}^\infty|2^kQ|^2 \mathop{\textup{ess sup}}_{\vec y\in (2^kQ)^2 \setminus (2^{k-1}Q)^2}|K(x,y_1,y_2)-K(z,y_1,y_2)| \\
&\le K_1,
\end{align*}where $\frac 12 Q$ is a cube which contains both $x$ and $z$ with $\ell(\frac 12 Q)=\|x-z\|_\infty$. However, the Dini condition in the $x$-variable can be written as the following 
\[
\sum_{k=1}^\infty \sup_{x,z}\omega^{x,z}(2^{-k})<\infty.
\]
It is hard to find an example to differentiate these two conditions. 
\end{Remark}
\begin{Remark}
We claim that actually the $L^r$-H\"ormander condition is strictly weaker than (H2). Indeed, (H2) is essentially of H\"older type while $L^r$ H\"ormander condition is essentially of Dini type. To prove our claim we borrow the example from \cite{MPT}. And to make things easier we only consider the linear case in one dimension. Define 
\[
K(x)=|x-4|^{-\frac 1{r'}}\Big(\log \frac e{|x-4|}\Big)^{-\frac {1+\beta}{r'}}\chi^{}_{\{3< x< 5\}}(x).
\]
It is easy to check that $K\in L^{r'}\cap L^1$. Then $T: f\rightarrow K*f$ is bounded on $L^p$ for all $1\le p \le \infty$. It is already proved in \cite{MPT} that $K$ satisfies the $L^r$-H\"ormander condition. Define 
\begin{equation*}
K_\ell(x)=
\begin{cases}
K(x), & x\in {\mathop\bigcup}_{k=0}^{2^{\ell+1}-1}(3+\frac{k}{2^\ell}, 3+\frac{3k+1}{3\cdot2^\ell}], \\
0, & \text{otherwise}. 
\end{cases}
\end{equation*} 
Similar argument as that in \cite{MPT} shows that $K_\ell$ satisfies the $L^r$-H\"ormander condition uniformly, i.e. $\sup_\ell (K_\ell)_r< \infty$. 
Let $x=0 $, $z=2^{-\ell-1}$ and $I_\ell=[0,  2^{-\ell})$. We need to analyze
\[
\Big(  \int_{S_{j}(I_\ell)}  |K_\ell(y-x)-K_\ell(y-z)|^{r'}  \d y \Big)^{\frac 1{r'}}.
\]
Observe that only $j=\ell+2$ and $j=\ell+3$ are non-zero terms. We have 
\begin{align*}
&\Big(  \int_{2^{\ell+2}I_\ell\setminus 2^{\ell+1}I_\ell}  |K_\ell(y)-K_\ell(y-2^{-\ell-1})|^{r'}  \d y \Big)^{\frac 1{r'}}
\\
&\ge \Big(\sum_{k=0}^{2^\ell-1}  \int_{3+\frac{k}{2^\ell}}^{3+\frac{3k+1}{3\cdot2^\ell}}  |K_\ell(y)|^{r'}  \d y \Big)^{\frac 1{r'}}\\
&\gtrsim \|K\|_{L^{r'}}\gtrsim 2^{(\delta-\frac 1r)\ell}\|K\|_{L^{r'}}\frac{|x-z|^{(\delta-\frac 1r)}}{|I_\ell|^{\delta}}2^{- (\ell+2)\delta}. 
\end{align*}
This shows that there exist a sequence of operators $T_\ell: f\rightarrow K_\ell*f$, whose kernels satisfy $L^r$ H\"ormander condition uniformly. However, since $\delta>\frac 1r$, the constant $C$ in (H2) tends to infinity when $\ell\rightarrow \infty$, which means these two conditions cannot be equivalent. In other words, $L^r$-H\"ormander condition is strictly weaker than (H2). 
\end{Remark}

In the end, we give an application of our result for multilinear Fourier multipliers:
\[
T(f_1,\cdots, f_m)(x)= \int_{\bbR^{mn}}a(\vec y)e^{2\pi i x\cdot (\sum_{i=1}^m y_i)} \prod_{i=1}^m \widehat f(y_i)\d \vec y,
\]
it is shown in \cite{BD} that multilinear Mihlin condition implies the assumption (H2). However, for the multilinear H\"ormander condition (\cite{GS}), i.e.,
 \begin{equation}\label{eq:mh}
 \sup_{R>0}\|a(R\xi)\chi_{\{1< |\xi|<2\}}\|_{H^s(\bbR^{m n})}<\infty,\quad \frac {mn}2<s\le mn,
 \end{equation}which is weaker than multilinear Mihlin condition, it is unknown. Very recently, Chaffee, Torres and  Wu \cite{CTW} showed that \eqref{eq:mh} implies the multilinear $L^r$-H\"ormander condition with $r=mn/s$. Therefore, Theorem \ref{thm:m2} also applies to multilinear Fourier multipliers whose symbols satisfy \eqref{eq:mh}.

As a consequence of our sparse domination theorem, we can give quantitative weighted bounds for them. We have
\begin{Theorem}\label{thm:multipliers}
Let $T$ satisfy the assumption in Theorem \ref{thm:m2}, then 
 \[
 \|T\|_{L^{p_1}(w_1)\times\cdots\times L^{p_m}(w_m)\rightarrow L^p(v_{\vec w})}\le c_{m,n,T,\vec P}[\vec w]_{A_{\vec P/r}}^{\max\{1, \max_i\frac{(p_i/r)'}{p}\}}.
 \]
\end{Theorem} 
Here 
\[
[\vec w]_{A_{\vec P/r}}:= \sup_Q \Big(\dashint_Q \prod_{i=1}^m w_i^{\frac p{p_i}}\Big) \prod_{i=1}^m\Big(\dashint_Q  w_i^{-\frac {r}{p_i-r}}\Big)^{\frac{p(p_i-r)}{p_ir}}.
\] 
 For the proof, we refer the readers to  \cite{BCDH}. One can also follow the maximal function trick used in \cite{DHL} and then utilize the result in \cite{LMS}. We can also obtain the  $A_p$-$A_\infty$ type bounds, see \cite{LS, DHL} for details. As we have discussed in the above, all these estimates apply to the multilinear Fourier multipliers with symbols satisfying \eqref{eq:mh}. Notice that the qualitative result was obtained by the author and Sun in \cite{LS1}.
 
 \section*{Acknowledgement}
The author would like to thank Prof. Tuomas Hyt\"onen for careful reading an earlier version of this paper and valuable suggestions which have improved the quality of this paper. Thanks also go to the anonymous referees for valuable suggestions.

\end{document}